\documentclass[10pt,conference]{ieeeconf}

\usepackage{mathrsfs}
\usepackage{graphics} 
\usepackage{amsmath} 
\usepackage{amssymb}  

\usepackage{amsthm}
\usepackage{cite}
\usepackage{bm}
\usepackage{acronym}
\usepackage{paralist}
\usepackage{float}
\usepackage{color}
\usepackage{epstopdf}
\usepackage{multicol}
\usepackage{tikz}
\usepackage{hyperref}
\usepackage{graphicx}
\usepackage{soul}
\usepackage{mathtools}
\usepackage{siunitx}
\usepackage{empheq}
\usepackage{accents}
\usepackage{subcaption}
\usepackage{mdframed}
\usepackage[export]{adjustbox}
\renewcommand{\theenumi}{\roman{enumi}}

\makeatletter
\hypersetup{colorlinks=true}
\AtBeginDocument{\@ifpackageloaded{hyperref}
  {\def\@linkcolor{blue}
  \def\@anchorcolor{red}
  \def\@citecolor{red}
  \def\@filecolor{red}
  \def\@urlcolor{red}
  \def\@menucolor{red}
  \def\@pagecolor{red}
\begingroup
  \@makeother\`%
  \@makeother\=%
  \edef\x{%
    \edef\noexpand\x{%
      \endgroup
      \noexpand\toks@{%
        \catcode 96=\noexpand\the\catcode`\noexpand\`\relax
        \catcode 61=\noexpand\the\catcode`\noexpand\=\relax
      }%
    }%
    \noexpand\x
  }%
\x
\@makeother\`
\@makeother\=
}{}}
\makeatother

\newtheorem{Theorem}{Theorem}
\newtheorem{Lemma}{Lemma}
\newtheorem{Problem}{Problem}
\newtheorem{Remark}{Remark}
\newtheorem{Corollary}{Corollary}

\newtheorem{Assumption}{Assumption}
\newtheorem{Definition}{Definition}

\DeclareMathOperator{\R}{\mathbb R}

\DeclareMathOperator*{\argmin}{arg\,min}

\newcommand{\bequ}{\begin{eqnarray}}
\newcommand{\eequ}{\end{eqnarray}}

\newcommand{\ubar}[1]{\underaccent{\bar}{#1}}

\IEEEoverridecommandlockouts

\def\BibTeX{{\rm B\kern-.05em{\sc i\kern-.025em b}\kern-.08em
    T\kern-.1667em\lower.7ex\hbox{E}\kern-.125emX}}

\begin{document}

\title{\LARGE \bf A Fixed-Time Stable Adaptation Law for Safety-Critical Control under Parametric Uncertainty}

\author{Mitchell Black \and Ehsan Arabi \and Dimitra Panagou
\thanks{
The authors would like to acknowledge the support of the National Science Foundation award number 1931982.}
\thanks{The authors are with the Department of Aerospace Engineering, University of Michigan, Ann Arbor, MI, USA; \texttt{\{mblackjr, earabi, dpanagou\}@umich.edu}.}
}
\maketitle


\begin{abstract}
We present a novel technique for solving the problem of safe control for a general class of nonlinear, control-affine systems subject to parametric model uncertainty. Invoking Lyapunov analysis and the notion of fixed-time stability (FxTS), we introduce a parameter adaptation law which guarantees convergence of the estimates of unknown parameters in the system dynamics to their true values within a fixed-time independent of the initial parameter estimation error. We then synthesize the adaptation law with a robust, adaptive control barrier function (RaCBF) based quadratic program to compute safe control inputs despite the considered model uncertainty. To corroborate our results, we undertake a comparative case study on the efficacy of this result versus other recent approaches in the literature to safe control under uncertainty, and close by highlighting the value of our method in the context of an automobile overtake scenario.
\end{abstract}


\section{Introduction}




Safe control design in our increasingly technology-reliant society is important, yet difficult even under idealized conditions. In theory, it may be performed by a variety of techniques, including control barrier function (CBF)-based strategies \cite{wielandallgower2007cbf,cortez2019cbfmechsys}. As a set-theoretic method for enforcing that the states of a system remain safe, i.e. within a specified set, CBFs have proven to be effective both in correcting some potentially unsafe control action \cite{rauscher2016constrained,Cheng2019EndToEndSafeRL} and as a constraint in optimization-based control \cite{black2020quadratic,garg2020fixedtime}. In fact, control design using quadratic programs (QP) as a means to synthesize control Lyapunov function (CLF)-based performance specifications, which can guarantee convergence of the closed-loop trajectories to some goal set, and CBF-based safety has become a staple in the literature \cite{ames2014control,ames2017control}. 

While CLFs encode a goal-reaching performance specification, the rate of convergence depends on the form of the condition imposed on its time-derivative. For example, while an exponentially stable CLF \cite{ames2014rapidly} guarantees that system trajectories reach a goal set as time tends toward infinity, the principle of fixed-time stability (FxTS) is used to define a fixed-time stable CLF \cite{garg2020fixedtime} which drives the closed-loop trajectories to the goal within a fixed-time independent of initial conditions. 

Robust approaches to safety \cite{xu2015robustness}, which often shrink the safe set in order to preserve set-invariance in the presence of worst-case disturbances, tend to be overly conservative. Typically, conservatism degrades the ability of the controller to meet performance objectives. In contrast, adaptive safety techniques \cite{Taylor2019aCBF} are prone to chattering control solutions, especially as the state approaches the boundary of the safe set. While a highly oscillatory control signal may suffice in simulation, it can destabilize dynamical systems in practice. In an approach which combines these two regimes, however, the authors of \cite{Lopez2020racbf} demonstrate that robust, adaptive control barrier functions (RaCBF) can mitigate these shortcomings. In addition, they utilize set-membership identification (SMID) to reduce the uncertainty surrounding a set of unknown parameters in the system dynamics, which results in improved performance. In a learning-based approach, \cite{taylor2020Learning} seeks to improve the performance of a Segway controller by learning unstructured uncertainty in the dynamics of the control barrier function itself. And while many of these techniques have been demonstrably effective in improving controller performance, the issue of navigating the domain near the boundary of the safe region remains problematic without formal guarantees of deciphering model uncertainty.

In considering some structured, parametric uncertainty for a class of nonlinear, control-affine systems, we introduce a novel parameter adaptation law which provides such a guarantee to drive the parameter estimates to their true values within fixed-time, independent of the initial estimates. Using Lyapunov analysis, we leverage this result to define an upper bound on the parameter estimation error as an explicit function of time. We then introduce a new condition on the time derivative of an RaCBF that guarantees forward-invariance of a \textit{shrunken} safe set, which approaches the \textit{nominal} safe set within the fixed-time horizon. In a case study on the ability of our proposed controller to tolerate marginally safe regions of the state space versus a selection from the literature, we compare the works of \cite{Taylor2019aCBF,Lopez2020racbf,black2020quadratic} and \cite{Zhao2020robustQP} on a simple problem with static obstacles. Finally, we use an automobile overtaking scenario to highlight another advantage of the proposed method: with the uncertain parameters in the system dynamics guaranteed to be known within fixed-time, our controller can accomplish a maneuver even under high levels of uncertainty.

The paper is organized as follows. Section II reviews set invariance via CBFs, finite-time stability (FTS), FxTS, and an FTS parameter estimation scheme from the literature. In Section III we formalize the problem at hand. Section IV contains our novel FxTS adaptation law and its implications on safe control under uncertainty. Section V highlights a simple case study used to compare other recent work to our proposed method, and its application on a highway overtake example. We conclude with a summary of contributions and directions for future work in Section VI.


\section{Mathematical Preliminaries}\label{sec: math prelim}

In the rest of the paper, $\mathbb R$ denotes the set of real numbers. The ones vector of size $n \times m$ is denoted $\mathbf{1}_{n \times m}$. We use $\|\cdot\|$ to denote the Euclidean norm and $\|\cdot\|_{\infty}$ to denote the $\mathscr{L}_{\infty}$ norm. As convention, we denote the minimum and maximum eigenvalue of a matrix $M$ as $\lambda_{min}(M)$ and $\lambda_{max}(M)$ respectively. We write $\partial S$ for the boundary of the closed set $S$, and $\textrm{int}(S)$ for its interior. The Lie derivative of a function $V:\mathbb R^n\rightarrow \mathbb R$ along a vector field $f:\mathbb R^n\rightarrow\mathbb R^n$ at a point $x\in \mathbb R^n$ is denoted as $L_fV(x) \triangleq \frac{\partial V}{\partial x} f(x)$. 

We now review the preliminaries, including a finite-time parameter adaptation law from the literature which inspired our fixed-time adaptation law. 


\subsection{Set Invariance}

In this paper, we consider the following class of nonlinear, control-affine systems subject to parametric uncertainty:
{\small
\begin{equation}\label{uncertain system}
    \begin{aligned}
        \dot x(t) &= f(x(t)) + g(x(t))u(t) + \Delta(x(t)) \theta, \;\;
        x(t_0) = x_0,
    \end{aligned}
\end{equation}
}\normalsize
\noindent where $x(t) \in \mathbb{R}^n$ denotes the state, $u(t) \in U \subset \mathbb{R}^m$ the control input, and $\theta \in \Theta \subset \mathbb{R}^p$ some constant, bounded, unknown parameters. We assume that $f: \mathbb{R}^n \rightarrow \mathbb{R}^n$ and $g: \mathbb{R}^n \rightarrow \mathbb{R}^{n \times m}$ are locally Lipschitz, that $\Delta : \mathbb{R}^n \rightarrow \mathbb{R}^{n \times p}$ is a known regressor matrix, and that $f$, $g$, and $\Delta$ are bounded for bounded inputs. The regressor $\Delta$ may capture, for example, the effects of sensing or modelling errors in the system dynamics, whereas the $\theta$ vector parameterizes such errors. We also define $\hat{\theta}$ as the estimated parameter vector, and $\Tilde{\theta} = \theta - \hat{\theta}$ as the parameter estimation error vector. As such, the parameter error dynamics are described by
\begin{equation}\label{error dynamics}
    \dot{\Tilde{\theta}} = \dot{\theta} - \dot{\hat{\theta}} = -\dot{\hat{\theta}}, \;\; \Tilde{\theta}(0) = \Tilde{\theta}_0.
\end{equation}

Let us also define a set of safe states as
\begin{subequations}\label{set}
    \begin{align}
        S &= \left \{ x \in \mathbb{R}^n : h(x) \geq 0 \right \}\\
        \partial S &= \left \{ x \in \mathbb{R}^n : h(x) = 0 \right \}\\
        \textrm{int}(S) &= \left \{ x \in \mathbb{R}^n : h(x) > 0 \right \}
    \end{align}
\end{subequations}

\noindent where $h: \mathbb{R}^n \rightarrow \mathbb{R}$ is continuously differentiable. Then, the following Lemma, known as Nagumo's Theorem \cite{blanchini1999set},
provides a necessary and sufficient condition on the forward invariance of the set \eqref{set} under the system dynamics \eqref{uncertain system}.

\begin{Lemma}
    Let a unique closed-loop solution to \eqref{uncertain system} exist in forward time.
    We say that the set $S$ can be \textit{rendered} forward-invariant if and only if there exists a control $u\in U$ such that
    \begin{equation}\label{eq: controlled forward invariance}
        \sup_{u \in U} \{L_fh(x) + L_gh(x)u + L_{\Delta}h(x)\theta\} \geq 0, \forall x \in \partial S.
    \end{equation}
\end{Lemma}
\begin{Remark}
    The closed-loop system \eqref{uncertain system} admits a unique solution if $u(t) \in U$ is Lipschitz.
\end{Remark}

The authors of \cite{Lopez2020racbf} ensure $S$ is safe with respect to the uncertain dynamics \eqref{uncertain system} by enforcing that a \textit{shrunken} set, $S_r$, is safe.
Before formally defining $S_r$ we make the following assumption.

\begin{Assumption}\label{ass bounded parameters}
    The set $\Theta$ to which the unknown parameters $\theta$ belong is known, compact, and convex.
\end{Assumption}

\begin{Remark}
    Assumption \ref{ass bounded parameters} imples that we can also restrict
    $\hat{\theta} \in \Theta$. Thus, we can define an upper bound $\vartheta \coloneqq \sup_{\theta_1,\theta_2 \in \Theta}(\|\theta_1 - \theta_2\|_{\infty})$ on the norm $\|\Tilde \theta\|_\infty$ of the parameter estimation error $\Tilde \theta$, so that $\|\Tilde{\theta}\|_{\infty} \leq \vartheta.$
\end{Remark}


We let $\boldsymbol{\vartheta} = \vartheta \cdot \mathbf{1}_{p\times1}$, and are now ready to present the definition for the \textit{shrunken} set, $S_r$, as defined in \cite{Lopez2020racbf}.
\begin{subequations}\label{set robust}
    \begin{align}
        S_r &= \left \{ x \in \mathbb{R}^n : h(x) \geq \frac{1}{2} \boldsymbol{\vartheta}^T \Gamma ^{-1} \boldsymbol{\vartheta} \right \}\\
        \partial S_r &= \left \{ x \in \mathbb{R}^n : h(x) = \frac{1}{2} \boldsymbol{\vartheta}^T \Gamma ^{-1} \boldsymbol{\vartheta} \right \}\\
        \textrm{int}(S_r) &= \left \{ x \in \mathbb{R}^n : h(x) > \frac{1}{2} \boldsymbol{\vartheta}^T \Gamma ^{-1} \boldsymbol{\vartheta} \right \}
    \end{align}
\end{subequations}

\noindent where $\Gamma$ is a constant, positive-definite matrix such that $h(x(0)) \geq \frac{1}{2}\boldsymbol{\vartheta}^T \Gamma^{-1} \boldsymbol{\vartheta}$.
A new CBF may be defined as
\begin{equation}\label{RaCBF}
    h_r(x,\boldsymbol{\vartheta}) = h(x) - \frac{1}{2} \boldsymbol{\vartheta}^T \Gamma ^{-1} \boldsymbol{\vartheta},
\end{equation}

\noindent and the sufficient condition that renders $S_r$ safe is
\begin{equation}\label{robust forward-invariance}
    \dot{h}_r = \frac{\partial h}{\partial x} \dot{x} - \boldsymbol{\vartheta}^T \Gamma ^{-1} \dot{\boldsymbol{\vartheta}} \geq -\alpha (h_r),
\end{equation}

\noindent where $\alpha: [0,a) \rightarrow [0,\infty)$ is some class $K$ function, i.e., strictly increasing with $\alpha(0) = 0$. 

\begin{Remark}
    If $S_r \subset S$ in \eqref{set robust} is safe, then $S$ in \eqref{set} is safe. Thus, satisfaction of \eqref{robust forward-invariance} implies that \eqref{eq: controlled forward invariance} holds.
\end{Remark}

\subsection{Finite- and Fixed-Time Stability}

We now address performance criteria: finite- and fixed-time stability of an equilibrium point of the dynamical system
\begin{equation}\label{nonlinear system}
    \begin{aligned}
        \dot x(t) &= f(x(t)), \quad x(t_0) = x_0,
    \end{aligned}
\end{equation}

\noindent for which it is assumed that a unique solution exists, where $x \in \mathbb{R}^n$, $f:\mathbb{R}^n \rightarrow \mathbb{R}^n$ is continuous, and $f(0) = 0$. 


\begin{Definition}[\hspace{-0.3pt}\cite{bhat2000finite}]\label{FTS Def}
    The origin of \eqref{nonlinear system} is globally finite-time stable (FTS) if the following conditions hold:
    \begin{itemize}
        \item The origin of \eqref{nonlinear system} is stable in the sense of Lyapunov
        \item Any solution $x(t,x_0)$ of \eqref{nonlinear system} reaches the origin in finite settling-time, $T(x_0)$, i.e. $x(t,x_0) = 0$, $\forall t \geq T(x_0)$.
    \end{itemize}
\end{Definition}

\begin{Definition}[\hspace{-0.3pt}\cite{polyakov2012nonlinear}]\label{FxTS Def}
    The origin of \eqref{nonlinear system} is fixed-time stable (FxTS) if it is globally FTS and any solution $x(t,x_0)$ of \eqref{nonlinear system} reaches the origin in finite settling-time, $T$, independent of $x_0$, i.e. $x(t,x_0) = 0$, $\forall t \geq T$.
\end{Definition}



We refer the interested reader to \cite{bhat2000finite} and \cite{polyakov2012nonlinear} respectively for Lyapunov conditions which guarantee FTS and FxTS of the origin of \eqref{nonlinear system}.



\subsection{Finite-Time Parameter Estimation}

The parameter estimation scheme in the discussion to follow is predicated on Assumption \ref{ass state and control access}.

\begin{Assumption}\label{ass state and control access}
    The state, $x$, and control input $u$, of \eqref{uncertain system} are bounded, and $x$ is accessible for measurement.
\end{Assumption}


\noindent Let us now review the notion of persistent excitation (PE).

\begin{Definition}[\hspace{-0.3pt}\cite{Bodson1989AdaptiveBook}]\label{Def PE}
    A vector or matrix function, $\phi$, is persistently excited (PE) if there exist $T > 0$, $\epsilon > 0$, such that $\int_{t}^{t+T}\phi (r) \phi ^T (r) dr \geq \epsilon I, \forall t \geq 0$.
\end{Definition}

\noindent Now we introduce the following assumption on $\Delta (x)$.

\begin{Assumption}\label{ass delta PE}
    The transpose of the regressor matrix in \eqref{uncertain system}, $\Delta^T (x)$,  is persistently excited.
\end{Assumption}

\begin{Remark}
    Positive-definiteness of $\Delta^T\Delta$ is sufficient for $\Delta^T$ to satisfy the PE condition.
\end{Remark}

\noindent We now review a FTS parameter estimation scheme which forms the basis for our FxTS adaptation law and was first proposed for nonlinear dynamical systems in \cite{Na2011nonlinear}, and extended for robotic applications in \cite{Na2015robotic}. First, we note that we may re-write \eqref{uncertain system} as:
\begin{equation}\label{Na uncertain system}
    \begin{aligned}
        \dot x(t) &= \varphi (x,u) + \Phi (x) \theta, \\
    \end{aligned}
\end{equation}


\noindent where $\varphi(x,u) = f(x) + g(x)u$ and $\Phi(x) = \Delta(x)$.
The authors of \cite{Na2011nonlinear} and \cite{Na2015robotic} introduce $x_f$, $\varphi _f$, and $\Phi _f$ to filter $x$, $\varphi$, and $\Phi$ as follows:
\begin{align}
    k\dot{x}_f + x_f &= x, \quad x_f(0) = 0, \quad \dot{x}_f(0) = 0 \label{xf1} \\
    k\dot{\varphi}_f + \varphi _f &= \varphi, \quad \varphi _f(0) = 0, \quad \dot{\varphi}_f(0) = 0 \label{phif1} \\
    k\dot{\Phi}_f + \Phi _f &= \Phi, \quad \Phi _f(0) = 0, \quad \dot{\Phi}_f(0) = 0 \label{phif2}
\end{align}



\noindent where $k>0$ is a design parameter. Mimicking the form of the system dynamics, the filtered system dynamics are:
\begin{align}\label{Filtered System}
    \dot x_f = \varphi _f + \Phi _f\theta,
\end{align}

\noindent which serves as the basis for estimating $\theta$. 
We now observe that $\dot{x}_f - \varphi_f = \Phi _f\theta$ and define an auxiliary and integrated regressor matrix $P$ and vector $Q$ such that:
\begin{align}
    \dot{P} &= -\ell_e P + \Phi_f ^T \Phi_f, \quad &P(0) &= 0 \label{filter P} \\
    \dot{Q} &= -\ell_e Q + \Phi_f ^T (\dot{x}_f - \varphi_f), \quad &Q(0) &= 0 \label{filter Q} 
\end{align}

\noindent where $\ell_e>0$ is another design parameter. The solutions are:
\begin{align}
    P(t) &=\int_{0}^{t} e^{-\ell_e(t-r)}\Phi_f ^T(r) \Phi_f (r) dr \label{sol P} \\
    Q(t) &=\int_{0}^{t} e^{-\ell_e(t-r)}\Phi_f ^T(r) (\dot{x}_f(r)-\varphi_f (r)) dr \label{sol Q}
\end{align}

\noindent and from them it may be discerned that $Q = P\theta$. Now, define an additional auxiliary vector as
 \begin{equation}\label{filter W}
    W = P\hat{\theta} - Q = -P\Tilde{\theta}.
\end{equation}
 
\noindent Then, the authors of \cite{Na2015robotic} introduce their adaptation law as
\begin{equation}\label{FT estimation law}
    \dot{\hat{\theta}} = -\Lambda \frac{P^TW}{\|W\|},
\end{equation}
 
\noindent where $\Lambda$ is a constant, positive definite, gain matrix. Finite-time (FT) convergence of the estimated parameters to their true values is guaranteed by the following result.

\begin{Theorem}[\hspace{-0.3pt}\cite{Na2015robotic}]\label{Thm FT Parameter Estimation}
    For system \eqref{Na uncertain system} with parameter adaptation law \eqref{FT estimation law} and $\lambda_{min} (P) > \sigma > 0$, the parameter estimation error $\Tilde{\theta}$ converges to zero in finite-time $t_a$, satisfying $t_a \leq \|\Tilde{\theta}(0)\|\frac{\lambda _{max}(\Gamma ^{-1})}{\sigma}$.
\end{Theorem}

While in general it is not required that $\Lambda$ from \eqref{FT estimation law} and $\Gamma$ from \eqref{set robust} be equivalent, we consider this to be true for the rest of the paper. In Section \ref{sec: Main Results}, we advance the FT parameter estimation result by proposing a FxTS adaptation law for the class of systems described by \eqref{uncertain system}.
 






\section{Problem Formulation}

We now formalize the problem under consideration.


\begin{Problem}\label{Problem Statement}
    Consider a nonlinear, control-affine dynamical system subject to parametric uncertainty as in \eqref{uncertain system}. Given that Assumptions \ref{ass bounded parameters}-\ref{ass delta PE} hold, design an adaptation law, $\dot{\hat{\theta}}$, and controller, $u$, such that the following conditions are satisfied: 
    \begin{enumerate}
        \item The parameter estimation error converges to zero within fixed-time, $T_{\theta}$, i.e. $\Tilde{\theta}(t) \rightarrow 0$ as $t \rightarrow T_{\theta} < \infty$, independent of $\Tilde{\theta}(0)$.
        \item The system trajectories remain safe for all time, i.e. $x(t) \in S$,  $\forall t \geq t_0$.
        \item The system trajectories converge to a goal set within fixed-time, $T_g$, i.e. $x(t) \in S_g$, $\forall t \geq T_g$.
    \end{enumerate}
\end{Problem}

The following section contains our proposed method. 

\section{Main Results}\label{sec: Main Results}

Before introducing one of the main results of the paper, we address modifications to the filtering scheme \eqref{xf1}-\eqref{phif2}. In place of the first-order scheme of \cite{Na2011nonlinear},\cite{Na2015robotic}, we use the following second-order filters:
\begin{align}
    k_e^2\ddot{x}_f + 2k_e\dot{x}_f + x_f &= x \label{xf1 second} \\
    k_e^2\ddot{\varphi}_f + 2k_e\dot{\varphi}_f + \varphi _f &= \varphi \label{phif1 second} \\
    k_e^2\ddot{\Phi}_f + 2k_e\dot{\Phi}_f +  \Phi _f &= \Phi \label{phif2 second}
\end{align}

\noindent where all initial conditions are zero, i.e. $\beta_f(0) = 0$, $\dot{\beta}_f(0) = 0$, $\ddot{\beta}_f(0) = 0$, $\forall \beta \in \{x,\varphi,\Phi\}$. This second-order system is stable, strictly proper, and minimum-phase similarly to \eqref{xf1}-\eqref{phif2}, and in addition, it is critically damped with a natural frequency of $\omega_n = 1/k_e$. This is desirable, as critically damped systems exhibit the smallest settling time, $t_s$, without oscillations \cite{roskilly2015marine}. 

\subsection{FxTS Adaptation Law}

We now introduce one of the main results of the paper, an adaptation law which renders the trajectories $\tilde \theta(t)$ of the parameter estimation error fixed-time stable to zero, and thus guarantees convergence of the parameters to their true values within fixed-time.

\begin{Theorem}\label{Thm: FxT Parameter Adaptation}
    Consider a nonlinear, control-affine system with parametric uncertainty as in \eqref{uncertain system}. If \eqref{xf1 second}-\eqref{phif2 second} filter $x$, $\varphi$, and $\Phi$, and the auxiliary matrix P and vectors Q, W are defined by \eqref{filter P}-\eqref{filter W}, then, under the ensuing adaptation law
    \begin{equation}\label{FxT Estimation Law}
        \dot{\hat{\theta}} = \Gamma W \left(W^T P^{-T} W \right)^{-1} \left(-c_{1e} \nu ^{\gamma _{1e}} - c_{2e} \nu ^{\gamma _{2e}} \right),
    \end{equation}
    the estimated parameters, $\hat{\theta}(t)$, converge to the true parameters, $\theta$, in fixed-time, $T_{\theta}$, i.e.,  $\Tilde{\theta}(t) \rightarrow 0$ and $\hat{\theta}(t) \rightarrow \theta$ as $t \rightarrow T_{\theta}$, where
    \begin{equation}\label{parameter convergence time}
        T_{\theta} \leq T_b = \frac{1}{c_{1e} (1 - \gamma _{1e})} + \frac{1}{c_{2e} (\gamma _{2e} - 1)}.
    \end{equation}
    with $\nu = \frac{1}{2}W^T P^{-T} \Gamma ^{-1} P^{-1} W$, $c_{1e} > 0$, $c_{2e} > 0$, $0 < \gamma _{1e} < 1$, $\gamma _{2e} > 1$, and $\Gamma \in \R^{p\times p}$ being a constant, positive-definite, gain matrix.
\end{Theorem}

\begin{proof}
Consider the Lyapunov function candidate $V_{\Tilde{\theta}} = \frac{1}{2} \Tilde{\theta}^T \Gamma ^{-1} \Tilde{\theta}$ for the system of the parameter-error dynamics \eqref{error dynamics}. Since $\dot{\theta} = 0$, its time derivative along the trajectories of \eqref{error dynamics} reads $\dot{V}_{\Tilde{\theta}} = -\Tilde{\theta}^T \Gamma ^{-1} \dot{\hat{\theta}}$.
Applying the adaptation law \eqref{FxT Estimation Law} yields $\dot{V}_{\Tilde{\theta}} = -\Tilde{\theta} ^T W \left( W^T P^{-T} W\right)^{-1} \left(-c_{1e} \nu ^{\gamma _{1e}} - c_{2e} \nu ^{\gamma _{2e}} \right)$. Then, by substituting \eqref{filter W}, we obtain
\begin{align}
    \dot{V}_{\Tilde{\theta}} = -c_{1e} V_{\Tilde{\theta}} ^{\gamma _{1e}} - c_{2e} V_{\Tilde{\theta}} ^{\gamma _{2e}},
\end{align}
i.e., the fixed-time stability condition from \cite[Lemma 1]{polyakov2012nonlinear}. Hence, the origin of \eqref{error dynamics} is fixed-time stable, and the trajectories $\Tilde \theta(t)$ reach the origin within a settling time $T_\theta$, given by \eqref{parameter convergence time}. Consequently, the estimated parameter vector, $\hat{\theta}(t)$, converges to the true parameter vector, $\theta$, within a fixed time, $T_{\theta}$, i.e., $\hat \theta(T_\theta)=\theta$.
\end{proof}

Whereas previous studies (e.g. \cite{Chen2019FxTadaptationrobotic}) use
\eqref{filter P} and \eqref{filter Q} in the design of a FxT adaptation scheme that converges to some bounded set, \eqref{FxT Estimation Law} guarantees fixed-time convergence of the estimated parameters to their true values. With this knowledge, we derive an expression for the upper bound on the infinity norm of the parameter error as a function of time.

\begin{Corollary}\label{Cor: upper bound parameter error}
Let $\Tilde{\theta} = \theta - \hat{\theta}$ be the parameter estimation error vector associated with a system of the form \eqref{uncertain system}.
If all of the following conditions hold

    \vspace{5pt}\begin{enumerate}
        \item The estimated parameter update law, $\dot{\hat{\theta}}$, is given by \eqref{FxT Estimation Law}
        \item $u(x(t))$ is any locally Lipschitz controller
        \item $\Gamma$ is constant, positive-definite, and \textbf{diagonal}
        \item $\gamma_{1e} = 1 - \frac{1}{\mu_e}$ and $\gamma_{2e} = 1 + \frac{1}{\mu_e}$ for some $\mu_e > 1$
    \end{enumerate}\vspace{5pt}
then the following expression constitutes an upper bound on $\|\Tilde{\theta}(t)\|_{\infty}, \; \forall t \in [0,T_{\theta}]$:
\begin{align}\label{err bound}
    \|\Tilde{\theta}(t)\|_{\infty} \leq \sqrt{M\left ( \frac{1}{N} \tan{ \left [ \Xi - N\frac{c_{1e}}{\mu_e}t \right ] } \right )^{\mu_e}} \coloneqq \eta (t), 
\end{align}
where $M =  2 \lambda_{max}(\Gamma)$, $N = \sqrt{\frac{c_{2e}}{c_{1e}}}$, and 
\begin{equation}
    \Xi = \tan^{-1}{ \left (\frac{1}{2} N \boldsymbol{\eta}^T(0) \Gamma^{-1} \boldsymbol{\eta}(0)\right )}
\end{equation}
with $\boldsymbol{\eta}(t) = \eta(t) \cdot \mathbf{1}_{p\times1}$ and 
\begin{equation}\label{tighter t bound}
    T_{\theta} \leq \frac{\mu_e\Xi}{\sqrt{c_{1e}c_{2e}}} \leq T_b
\end{equation}

\end{Corollary}

\begin{proof}
    We consider the Lyapunov function candidate  $V_{\Tilde{\theta}} = \frac{1}{2}\Tilde{\theta} ^T \Gamma ^{-1} \Tilde{\theta}$, whose time derivative along the trajectories of \eqref{error dynamics}
    reads $\frac{dV_{\Tilde{\theta}}}{dt} = -c_{1e}V_{\Tilde{\theta}}^{1-1/\mu_e} - c_{2e}V_{\Tilde{\theta}}^{1+1/\mu_e}$.
    
    By separation of variables and integration, we now solve for $t$ as a function of $V(0)$ and $V(t)$. The change of variables $x = V_{\Tilde{\theta}}^{1/\mu_e}$ and $dx = \frac{1}{\mu_e}V_{\Tilde{\theta}}^{1 - 1/\mu_e}dV$ allows us to obtain:
    \small{
    \begin{align}
        &t = \int_{V(0)}^{V(t)} \frac{\mu_e x^{\mu_e -1} dx}{-c_{1e}x^{\mu_e-1}-c_{2e}x^{\mu_e+1}} =\\ \nonumber &-\frac{\mu_e}{\sqrt{c_{1e}c_{2e}}}\left [\tan^{-1}\left(NV_{\Tilde{\theta}}^{1/\mu_e}(t)\right) -  \tan^{-1}\left(NV_{\Tilde{\theta}}^{1/{\mu_e}}(0)\right) \right ]
    \end{align}} \normalsize
     where $N = \sqrt{\frac{c_{2e}}{c_{1e}}}$. This leads to 
    \small{
    \begin{align}\label{eq: error bound derivation}
        V_{\Tilde{\theta}}(t) 
        &= \left (\frac{1}{N} \tan\left [\tan^{-1}\left(N V(0)\right) -N\frac{c_{1e}}{\mu_e}t \right]\right)^{\mu_e}
    \end{align}}\normalsize
    where $V(0) = \frac{1}{2}\Tilde{\theta}(0)^T\Gamma^{-1}\Tilde{\theta}(0) \leq \frac{1}{2}\boldsymbol{\eta}(0)^T\Gamma^{-1}\boldsymbol{\eta}(0)$.
    Now, since $V_{\Tilde{\theta}} = \frac{1}{2}\Tilde{\theta} ^T \Gamma ^{-1} \Tilde{\theta}$, then with $\Gamma$ diagonal we can express $V_{\Tilde{\theta}} = \frac{1}{2}(\Gamma^{-1}_{11}\Tilde{\theta}_1^2 + ... + \Gamma^{-1}_{pp}\Tilde{\theta}_p^2)$, and observe that $V_{\Tilde{\theta}} \geq \frac{1}{2}\lambda^{-1}_{max}(\Gamma)\|\Tilde{\theta}\|^2  \geq \frac{1}{2}\lambda^{-1}_{max}(\Gamma)\|\Tilde{\theta}\|_{\infty}^2 $. Then, we substitute \eqref{eq: error bound derivation} in this inequality and rearrange terms to recover \eqref{err bound}.
    
    
    Then, for $0 \leq t \leq \frac{\mu_e}{c_{1e}N}\textrm{tan}^{-1}\left(\frac{1}{2}N\boldsymbol{\eta}^T\Gamma^{-1}\boldsymbol{\eta}\right)$ we have that \eqref{eq: error bound derivation} decreases monotonically 
    to zero. As such, we let $\Xi = \textrm{tan}^{-1} \left(\frac{1}{2}N\boldsymbol{\eta}^T\Gamma^{-1}\boldsymbol{\eta}\right)$ and obtain \eqref{tighter t bound}, which places a tighter upper bound on the settling time, $T_\theta$, than \eqref{parameter convergence time}.
\end{proof}


\begin{Remark}
    As a consequence of \eqref{err bound}, we may tighten the set of admissible parameters at time $t$ as $\Theta(t) = \{\theta \in \mathbb{R}^p: \; \|\theta - \hat{\theta}(t)\|_{\infty} \leq \eta (t)\}$,  $\forall t \in [0,T_{\theta}]$, and as $\Theta(t) = \{\hat{\theta}\}$, $\forall t \in (T_{\theta},\infty]$.
\end{Remark}


We now formalize the new RaCBF condition for forward-invariance of sets $S$ and $S_r$, defined in \eqref{set} and \eqref{set robust}.

\begin{Theorem}\label{Thm: Adaptive Safety}
    Let $\boldsymbol{\eta}(t) = \eta(t) \cdot \mathbf{1}_{p\times1}$ where $\eta(t)$ is given by \eqref{err bound}. Under the premises of Corollary \ref{Cor: upper bound parameter error}, the following condition is sufficient for forward-invariance of $S_r$, i.e. \eqref{set robust}:
    \small{
    \begin{align}\label{RaCBF Condition}
            &\sup_{u \in U}\{ L_fh(x) + L_gh(x)u + \Psi\} \nonumber \\
            &\geq -\alpha \left(h(x) - \frac{1}{2} \boldsymbol{\eta}(t)^T \Gamma^{-1} \boldsymbol{\eta}(t)) \right) + \mathrm{Tr}(\Gamma^{-1})\eta(t)\dot{\eta}(t)
    \end{align}}\normalsize
    where
    \small{
    \begin{align}
    \dot{\eta}(t) &= Z\tan^{\frac{\mu_e}{2}-1}\left(\Xi - N\frac{c_{1e}}{\mu_e}t\right)\sec^{2}\left(\Xi - N\frac{c_{1e}}{\mu_e}t\right)\label{eq: eta dot}\\
        \Psi &= \sum_{i=1}^p  \min\left\{C_iP_{\Theta}(\hat{\theta}_i - \eta),C_iP_{\Theta}(\hat{\theta}_i + \eta)\right\}\label{eq: safety projection}
    \end{align}}\normalsize
    for $Z = -c_{1e}\sqrt{MN^{2-\mu_e}}$, where $C_i$ denotes the $i^{th}$ column of $L_{\Delta}h(x)$ for $i \in \{1,\hdots,p\}$, and $P_{\Theta}$ the vector projection onto $\Theta$, as defined in \cite{Luenberger1969Optimization}.
    
\end{Theorem}

\begin{proof}
    We recall \eqref{RaCBF} and now replace the constant $\boldsymbol{\vartheta}$ with the time-varying $\boldsymbol{\eta}(t)$, where henceforth we drop the argument.
    Thus, $h_r(x,\boldsymbol{\eta}) = h(x) - \frac{1}{2}\boldsymbol{\eta}^T\Gamma^{-1}\boldsymbol{\eta}$, and $\dot{h}_r(x,\boldsymbol{\eta}) = \frac{\partial h_r}{\partial x}\dot{x} + \frac{\partial h_r}{\partial \boldsymbol{\eta}}\dot{\boldsymbol{\eta}}$. We will show that if \eqref{RaCBF Condition} holds then \eqref{robust forward-invariance} holds and thus \eqref{set robust} is forward-invariant. 
    
    First, $\frac{\partial h_r}{\partial x}\dot{x} = \frac{\partial h}{\partial x}\dot{x} = L_fh(x) + L_gh(x)u + L_\Delta h(x)\theta$, then with $\Gamma$ diagonal we can express $\frac{\partial h_r}{\partial \boldsymbol{\eta}}\dot{\boldsymbol{\eta}} = -\boldsymbol{\eta}^T\Gamma^{-1}\dot{\boldsymbol{\eta}} = -\mathrm{Tr}(\Gamma^{-1})\eta\dot{\eta}$, where
    we obtain $\dot{\eta}(t)$ in  \eqref{eq: eta dot} by differentiating \eqref{err bound} from Corollary \ref{Cor: upper bound parameter error}. 
    
    Next, we consider the case where $(\boldsymbol{\eta} + \hat{\theta}) \in \Theta$, for which $P_{\Theta}(\boldsymbol{\eta} + \hat{\theta}) = (\boldsymbol{\eta} + \hat{\theta})$. By \eqref{err bound} we have that $\hat{\theta}_i - \eta \leq \theta_i \leq \hat{\theta}_i + \eta$, $\forall i \in \{1,\hdots,p\}$, where $\Theta$ is convex by Assumption \ref{ass bounded parameters}. Thus, the solution of the following minimization problem represents the worst admissible effect of the unknown parameters on safety:
    \begin{equation}\label{eq: min phi}
        \phi^* = \argmin_{\phi \in \Theta} L_{\Delta}h(x)\phi
    \end{equation}
    This is a constrained linear program. As such, a unique minimizer, $\phi^* = [\phi_1^* \; \hdots \; \phi_p^*]^T$ exists, where by the fact that $L_\Delta h(x)\phi = \sum_{i = 1}^pC_i\phi_i$ with $C_i$ as the i$^{th}$ column of $L_{\Delta}h(x)$, we have that $\phi_1^*,\;\hdots\;,\phi_p^*$ are the minimizers of the following $p$ constrained linear programs: 
    \begin{equation}
        \phi_i^* = \argmin_{\hat{\theta}_i - \eta \leq \phi_i \leq \hat{\theta}_i + \eta}C_i\phi_i, \quad \forall i \in \{1,\hdots,p\}
    \end{equation}
    Furthermore, the solutions of constrained linear programs are guaranteed to be on the boundary of the solution domain, which in this case implies that either $\phi_i^* = \hat{\theta}_i - \eta$ or $\phi_i^* = \hat{\theta}_i + \eta$. Thus, 
    we denote $\Psi = L_{\Delta}h(x)\phi^*$ and recover \eqref{eq: safety projection} so that $\Psi \leq L_{\Delta}h(x)\theta$, $\forall \theta \in \Theta$. 
    
    For $(\boldsymbol{\eta} + \hat{\theta}) \notin \Theta$, \eqref{eq: min phi} is again solved, but now $P_{\Theta}(\boldsymbol{\eta} + \hat{\theta})$ reduces $\eta$ when necessary to enforce that $\phi \in \Theta$. We again obtain that a unique solution exists and that $\Psi \leq L_{\Delta}h(x)\theta$, $\forall \theta \in \Theta$. 
    
    For both cases we have that $L_fh(x) + L_gh(x)u + \Psi \leq \frac{\partial h}{\partial x}\dot{x}$ and thus \eqref{RaCBF Condition} implies \eqref{robust forward-invariance}, and \eqref{set robust} is forward-invariant.
\end{proof}

The use of the projection operator in Theorem \ref{Thm: Adaptive Safety} reduces the conservatism of the approach without compromising the robustness of the forward-invariance condition. 
\section{Case Study}\label{sec: case study}

\subsection{Comparing Controllers}

In the first numerical study, we investigate how our approach compares to other recent results in the literature, namely the adaptation laws from \cite{Taylor2019aCBF}, \cite{Lopez2020racbf}, and \cite{Zhao2020robustQP}, and the worst-case disturbance consideration of \cite{black2020quadratic}. As a basis for comparison, we consider a 2D single-integrator system subject to parametric uncertainty and challenge the controllers to 
safely achieve convergence to the origin by avoiding static obstacles separated by a small gap (Fig. \ref{Fig: Shoot the Gap Setup}). 

\begin{figure}[!ht]
    \centering
        \includegraphics[width=1\columnwidth,clip]{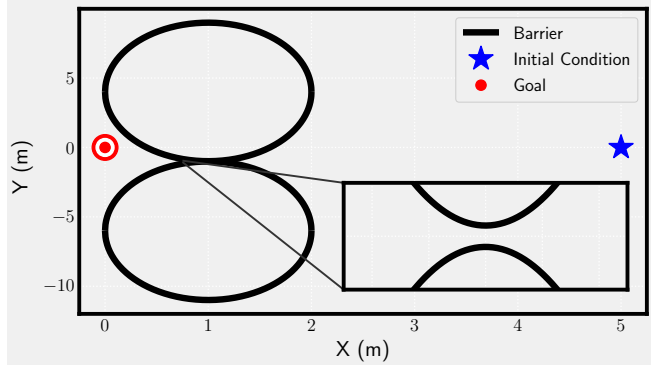}
    \caption{\small{Problem setup for the first numerical case study, "Shoot the Gap." The controller must determine what actions, $u_x$ and $u_y$, to take in order to realize safe trajectories from the Initial Condition to the Goal.}}\label{Fig: Shoot the Gap Setup}
    \vspace{-5mm}
\end{figure}

\subsubsection{Dynamics}

We denote $z = [x \; y]^T$ as the state, where $x$ and $y$ are the lateral and longitudinal position coordinates with respect to an inertial frame. The system dynamics are
\begin{align}\label{eq: simple dynamics}
    \dot{z} &= 
        \begin{bmatrix}
            1 & 0\\
            0 & 1
        \end{bmatrix}
        \begin{bmatrix}
            u_x\\
            u_y
        \end{bmatrix}
        +
        \Delta(z)
        \begin{bmatrix}
            \theta _1\\
            \theta _2
        \end{bmatrix},
\end{align}\normalsize

\noindent where the known regressor matrix is given by
\begin{align}
   \Delta(z) =  \nonumber K_{\Delta}\begin{bmatrix}
            1 + \sin^2(2\pi f_1 x) & 0 \\
            0 & 1 + \cos^2(2\pi f_2 y)
        \end{bmatrix}
\end{align}
\noindent with $\theta_1$, $\theta_2$ as constant parameters that are unknown a priori, and $K_{\Delta}$, $f_1$, $f_2$ given in Table \ref{table simple params}. Assumption \ref{ass bounded parameters} is enforced by defining lower and upper bounds $\ubar{\theta}$ and $\bar{\theta}$, respectively, and imposing $\ubar{\theta} \leq \theta_1,\theta_2 \leq \bar{\theta}$. The choice of $\Delta$ is such that $\Delta^T\Delta$ is positive-definite for all $z \in \R^2$, thus satisfying the PE condition and Assumption \ref{ass delta PE}.

\subsubsection{Control Formulation}

\begin{figure*}[!ht]
\hspace*{-0.20em}
\begin{subfigure}{0.34\textwidth}
    \includegraphics[width=1\textwidth,left]{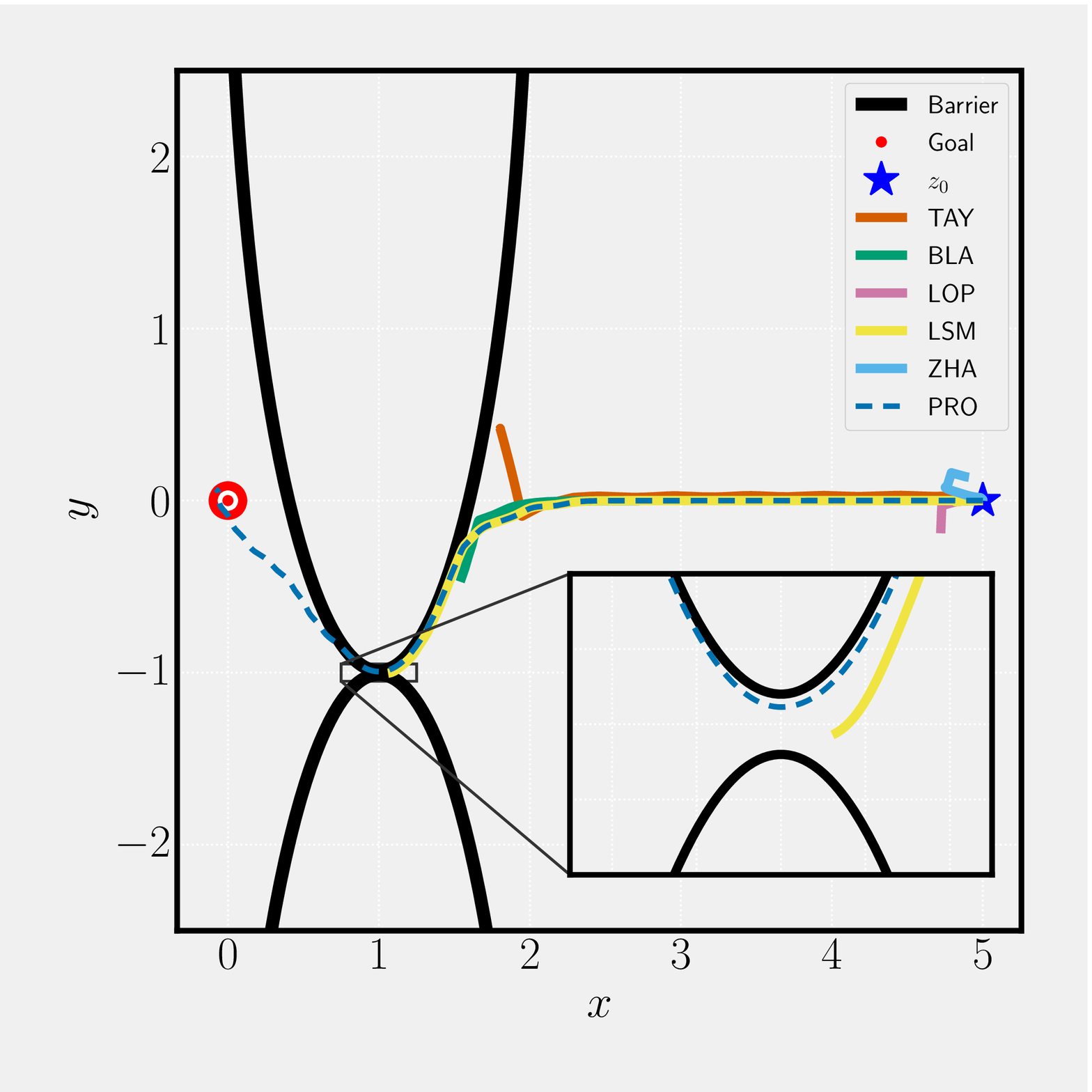}
    \caption{State trajectories}\label{fig: FxTS shoot the gap - trajectories}
\end{subfigure}%
\hspace*{-0.5em}
\begin{subfigure}{0.34\textwidth}
    \includegraphics[width=1\textwidth]{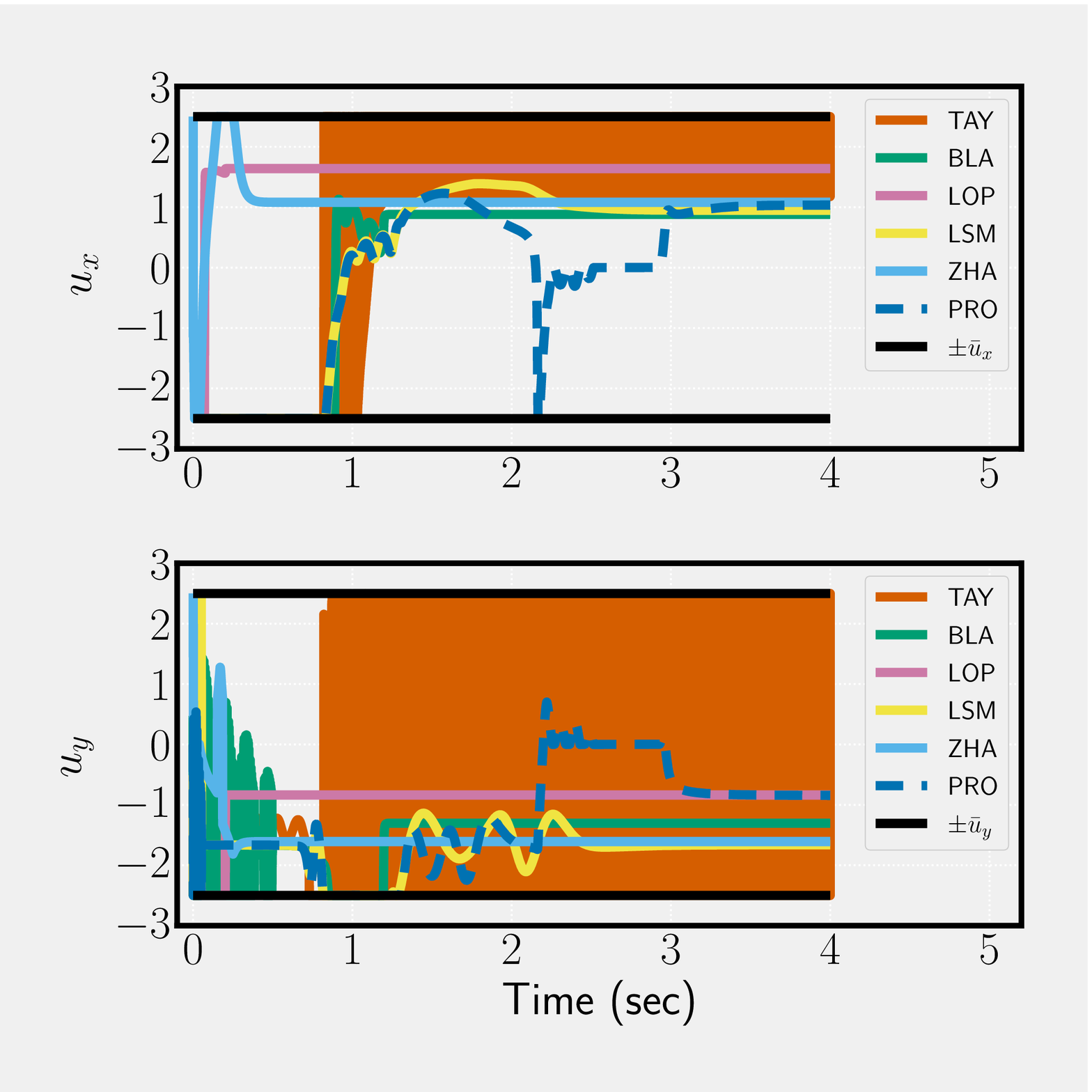}
    \caption{Control inputs}\label{fig: FxTS shoot the gap - controls}
\end{subfigure}%
\hspace*{-1.25em}
\begin{subfigure}{0.34\textwidth}
    \includegraphics[width=1\textwidth,right]{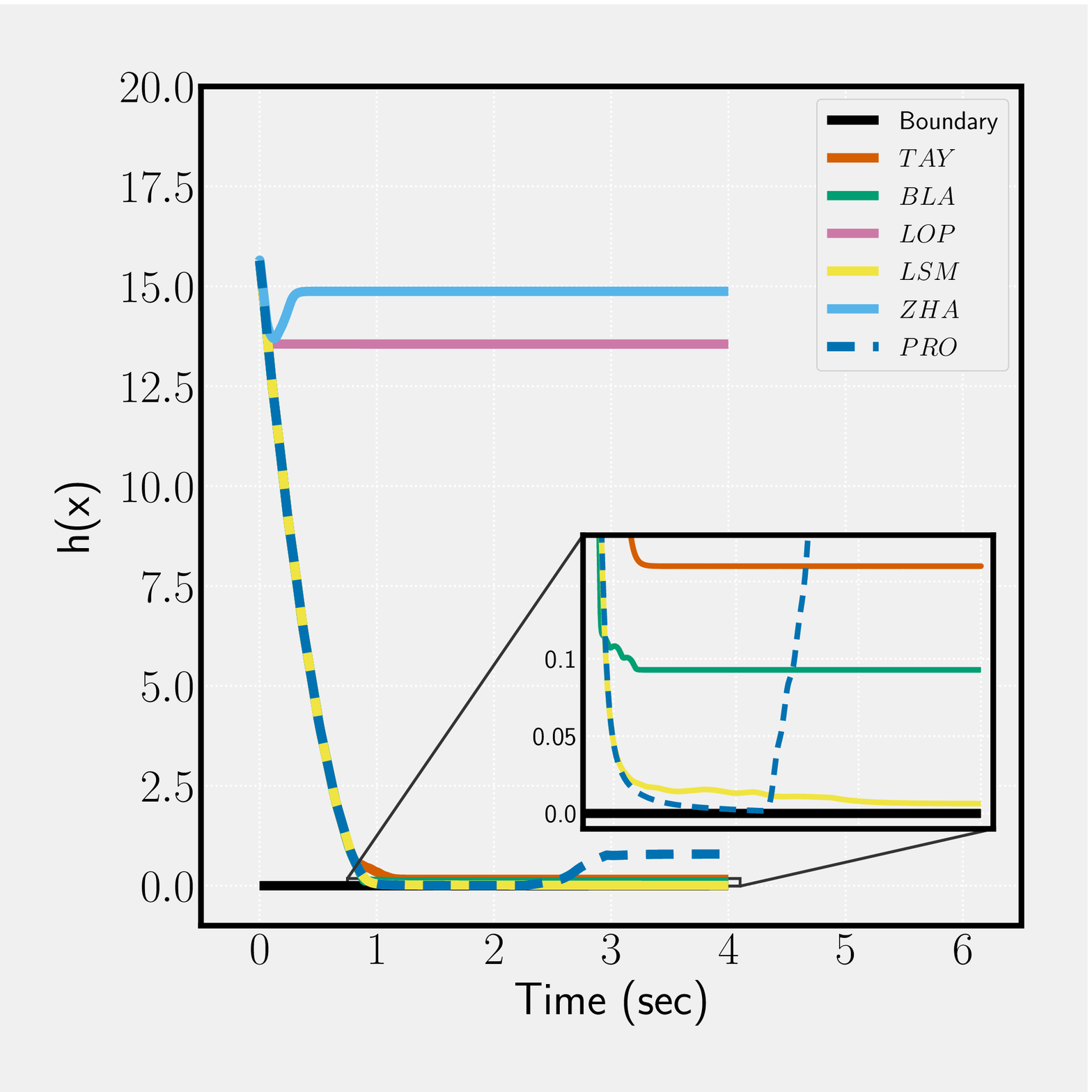}
    \caption{Value of level curve for safe set}\label{fig: FxTS shoot the gap - cbfs}
\end{subfigure}%
\caption{\small{State trajectories, control inputs, and control barrier function evolutions in time for the Shoot the Gap example.}}
\vspace{-5mm}
\end{figure*}

To encode the goal-convergence criterion we define the CLF:
\begin{align}\label{clf simple}
    V(z) = K_V(x^2 + y^2),
\end{align}
\noindent 
The safe states are those residing outside of the two ellipses shown in Figure \ref{Fig: Shoot the Gap Setup}, which results in the following two CBFs:
\begin{align}\label{cbfs simple}
    h_1(z) = \frac{(x - x_1)^2}{a^2} + \frac{(y - y_1)^2}{b^2} - 1 \\
    h_2(z) = \frac{(x - x_2)^2}{a^2} + \frac{(y - y_2)^2}{b^2} - 1
\end{align}

\noindent where $x_1$, $x_2$, $y_1$, $y_2$, $a$, and $b$ are parameters that define the location, size, and shape of the ellipses. 


We choose the CLF-CBF-QP framework (\hspace{-0.3pt}\cite{xu2015robustness,ames2014control}) for computing the control inputs. 
While we simulated the controllers from the literature both in their original form and with standardized FxT-CLFs to more fairly assess their abilities, no meaningful differences were observed in their ability to "shoot the gap." As such, we present results for the latter case.
Our control framework is then:
\small{
\begin{subequations}\label{CLF-CBF QP}
\begin{align}
    \min_{u,\delta _0,\delta _1,...,\delta_q} \frac{1}{2}u^{T}Qu &+p_0\delta _0^2 + \sum_{i=1}^q p_i\delta _i^2\\ 
    \nonumber &\textrm{s.t.} \\
    -\bar{u}_1 \leq &u_1 \leq \bar{u}_1 \label{ux constraint} \\
    -\bar{u}_2 \leq &u_2 \leq \bar{u}_2 \label{uy constraint} \\ 
    1 &\leq \delta_i \label{d1 constraint} \\
    \nonumber L_fV(z)+L_gV(z)u&+\phi(x,\Delta(x,t),\hat{\theta},\eta)\\
    \leq  \delta _0 &- c_1V(z)^{\gamma _1} - c_2V(z)^{\gamma _2} \label{pt_V}\\
    L_fh_i(z)+L_gh_i(z)u&+\psi(x,\Delta(x,t),\hat{\theta},\eta) \geq -\delta _ih_i(z) \label{pt_S1}
\end{align}
\end{subequations}}\normalsize

\noindent $\forall i \in \{1,\hdots,q\}$, where generally $u = [u_1 \; u_2]^T$ and for this problem $u_1=u_x$ and $u_2=u_y$, $\delta_0$ is a relaxation parameter on the performance objective whose inclusion guarantees feasibility of the QP, $\delta_i$ allows for larger negative values of $\dot{h}_i(z)$ away from the boundary of the safe set, and $p_i$ penalizes values of $\delta_i$, $\forall i \in \{0,...,q\}$.
The functions $\phi: \mathbb{R}^n \times \mathbb{R}^{n\times p} \times \mathbb{R}^p \times \mathbb{R}^p \rightarrow \mathbb{R}$ and $\psi: \mathbb{R}^n \times \mathbb{R}^{n\times p} \times \mathbb{R}^p \times \mathbb{R}^p \rightarrow \mathbb{R}$ represent the terms specific to the way each respective controller handles the uncertainty in the system dynamics. While all of \eqref{ux constraint}-\eqref{pt_S1} are linear in the decision variables, \eqref{ux constraint} and \eqref{uy constraint} enforce input constraints, \eqref{d1 constraint}
prevents over-conservatism in enforcing safety, \eqref{pt_V} encodes FxT convergence to the goal, and safety is guaranteed by \eqref{pt_S1}.

\subsubsection{Results}


The full set of parameters for this numerical case study\footnote{\label{github}Simulation code is accessible at \href{https://github.com/6lackmitchell/FxT_AdaptationLaw_ParametricUncertainty}{Github: https://tinyurl.com/y3xhylug}.} are provided in Table \ref{table simple params}.
{\small
\begin{table}[h!]
    \setlength{\tabcolsep}{4pt}
    \renewcommand{\arraystretch}{1.25}
    \centering
    \caption{Shoot the Gap Parameters}\label{table simple params}
    \begin{tabular}{|c|c|c|c|c|c|c|c|c|c|}
        \hline
        $\dot{x}$ & Val & QP & Val & CBF & Val & CLF & Val & $\dot{\hat{\theta}}$ & Val \\ \hline
        $f_1$ & 1  & $Q$ & $I_{2\times2}$ & $a$ & 1 & $K_V$ & 1 & $k_e$ & 0.001 \\ \hline
        $f_2$ & 4 & $p_0$ & 50 & $b$ & 4.99 & $T$ & 4 & $T_e$ & 0.2 \\ \hline
        $\theta_1$ & -1 & $p_1$ & 5 & $x_1$ & 1 & $\mu$ & 5 & $\mu_e$ & 5 \\ \hline
        $\theta_2$ & 1 & $p_2$ & 5  & $x_2$ & 1 & $c_1$ & 1.963 & $c_{1e}$ & 50 \\ \hline
        $\bar{\theta}_1$ & 10 & $\bar{u}_1$ & 2.5 & $y_1$ & -6 & $c_2$ & 1.963 & $c_{2e}$ & 50  \\ \hline
        $\bar{\theta}_2$ & 10 & $\bar{u}_2$ & 2.5 & $y_2$ & 4 & $\gamma_1$ & 0.8 & $\gamma_{1e}$ & 0.8  \\ \hline
        $K_{\Delta}$ & 0.833 & & & & & $\gamma_2$ & 1.2 & $\gamma_{2e}$ & 1.2 \\ \hline
        & & & & & & & & $\ell_e$ & 100\\
        \hline 
    \end{tabular}
    \vspace{-4mm}
\end{table}}

We endeavor to demonstrate that by learning the true values of the uncertain parameters in the system dynamics of \eqref{eq: simple dynamics}, our method is capable of approaching the boundary of the safe set more closely than previous results in the literature and, as a consequence, able to reach a goal which may require such an approach despite uncertainty. Table \ref{tab: key other works} provides the legend codes used to refer to these other works.

\begin{table}[h!]
    \centering
    \caption{Controllers from the Literature}\label{tab: key other works}
    \begin{tabular}{|c|c|c|c|c|c|c|c|}
        \hline
        Authors & Citation & Legend Code \\ \hline
        Taylor et al. & \cite{Taylor2019aCBF} & TAY \\ \hline
        Black et al. & \cite{black2020quadratic} & BLA \\ \hline
        Lopez et al. & \cite{Lopez2020racbf} (w/o SMID) & LOP \\ \hline
        Lopez et al. & \cite{Lopez2020racbf} (w/ SMID) & LSM \\ \hline
        Zhao et al. & \cite{Zhao2020robustQP} & ZHA \\ \hline
        Proposed Method &  & PRO \\ \hline
    \end{tabular}\par
    \vspace{2mm}
    Note: \cite{Lopez2020racbf} presents RaCBF-based control formulations with and without SMID for parameter estimation. We have considered both cases.
    \vspace{-3mm}
\end{table}


First, we observe that in accordance with Theorem \ref{Thm: FxT Parameter Adaptation}, Figure \ref{fig: theta hats} highlights that the parameter estimates, $\hat{\theta}$, do in fact converge to their true values within fixed-time $T_{\theta}$ given by \eqref{tighter t bound}.
\begin{figure}[!h]
    \centering
        \includegraphics[width=1\columnwidth,clip]{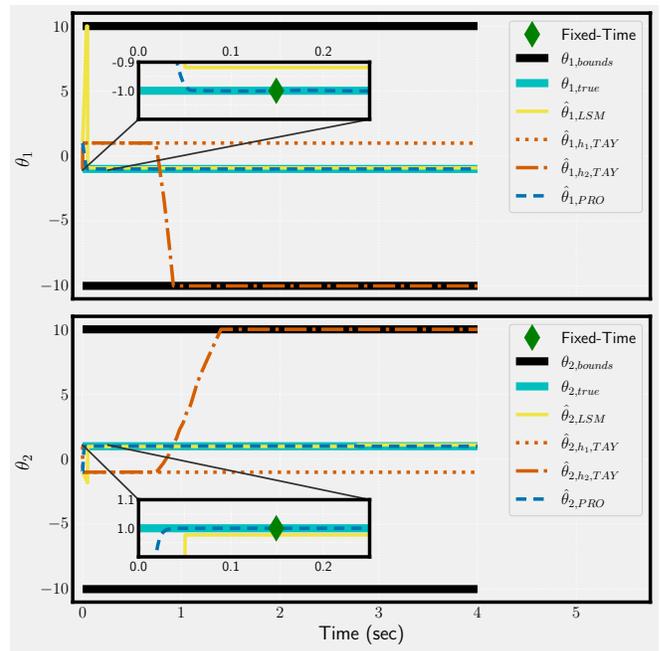}
    \caption{\small{Estimates of the unknown model parameters, $\theta$.}}\label{fig: theta hats}
    \vspace{-5mm}
\end{figure}
Figure \ref{fig: FxTS shoot the gap - trajectories} shows that our proposed method "shoots the gap" where the others do not;
that is, our method can tolerate regions of the state space which exist in close proximity to the boundary of the safe region. As such, it fulfills its specification of FxT convergence to the origin.
In this sense, our synthesized adaptation law and RaCBF-based controller is less restrictive than the existing literature.




\subsection{Highway Overtake}\label{sec: overtake}

We now consider an automobile highway overtake problem, similarly to \cite{black2020quadratic}, and show how our control formulation can guarantee success of the overtake maneuver under uncertainty, where the robust CBF approach cannot.

\subsubsection{Dynamics}

Just as in \cite{black2020quadratic}, we model the vehicles as kinematic bicycles using the model from \cite{Rajamani2012VDC}. Accordingly, the state vectors are $z_i = [x_i\;y_i\;\theta_i\;v_i]^T$, where $x$ and $y$ are planar Euclidean coordinates (longitudinal and transverse), $\theta$ is the heading angle, $v$ is velocity, and the subscript $i \in \{e,l\}$ denotes belonging to the Ego or Lead vehicle. The corresponding dynamical system is described by:
\small{
\begin{align}\label{eq: overtake_dynamics}
    \dot z_i = 
        \begin{bmatrix}
            v_i \ \cos(\theta_i) \\
            v_i \ \sin(\theta_i) \\
            0 \\  0
        \end{bmatrix}
        + 
        \begin{bmatrix}
            0 & 0 \\
            0 & 0 \\
            1 & 0 \\
            0 & 1 / M
        \end{bmatrix}
        \begin{bmatrix}
            \omega_i \\ a_i\end{bmatrix}+ \Delta_i(z)\theta,
\end{align}
}\normalsize

\noindent where $M$  is the mass of the vehicle in kg, and $\omega_i$ and $a_i$ represent the angular velocity and heading acceleration control inputs, for which the bounds $\ubar{\omega} \leq \omega_i \leq \bar{\omega}$ in rad$^{-1}$ and $\ubar{a} \leq a_i \leq \bar{a}$ in m/s$^2$ hold. For reference, all overtake parameter values may be found in Table \ref{table overtake params}\footnote{\label{mustang}$M$, $l_c$, and $w_c$ taken from the \href{https://www.ford.com/cars/mustang/models/shelby-gt350r}{2020 Ford Mustang Shelby GT: https://tinyurl.com/yxhn63of}.}.
We elect to model erratic, or distracted, driver behavior by the addition of the uncertain term $\Delta_i(x)\theta$, where $\Delta_i: \mathbb{R}^{4} \rightarrow \mathbb{R}^4 \times \mathbb{R}^2$ is the known regressor matrix. As such, we let $\Delta_e = \boldsymbol{0}_{n \times p}$ and $\Delta_l = \boldsymbol{0}_{n \times p}$ with the exception of $\Delta_{l,(0,0)}(z) = 1 + \frac{1}{2}(1 - \cos(2\pi f_{l,1}x_l))$ and $\Delta_{l,(1,1)}(z) = \frac{1}{10} + \frac{1}{20}(1 - \sin(2\pi f_{l,2}x_l))$.


\subsubsection{Problem Formulation}

We define the safe sets as:

\begin{equation}
    S_i = \{z \; | \; h_i(z) \geq 0\}, \; \forall i \in \{1,2,3\}
\end{equation}
where
\begin{align}
    h_1(z) &= K_s\left((y_e - E_R(z))(E_L(z) - y_e) \right)\label{cbf road}\\
    h_2(z) &= L - v\label{cbf speed}\\
    h_3(z) &= \left( \frac{x_e - x_l}{s_{x}}\right)^2 + \left( \frac{y_e - y_l}{s_{y}}\right)^2 - 1\label{cbf car}
\end{align}
and
\small{
\begin{align}
    E_R(z) = e_r &+ \frac{\theta_ev_e\sin(\theta_e)}{\bar{\omega}} \nonumber \\
    &- \frac{\theta_e^2}{2\bar{\omega}^2}\left(\frac{\bar{a}\sin(\theta_e)}{M} + v_e\bar{\omega}\cos(\theta_e)\right)\label{edge1}\\
    E_L(z) = e_l &- \frac{\theta_ev_e\sin(\theta_e)}{\bar{\omega}} \nonumber\\
    &- \frac{\theta_e^2}{2\bar{\omega}^2}\left(\frac{\bar{a}\sin(\theta_e)}{M} - v_e\bar{\omega}\cos(\theta_e)\right)\label{edge2}
\end{align}}\normalsize
where $e_r$ and $e_l$ denote the physical edges of the right and left side of the road such that \eqref{edge1} and \eqref{edge2} imply that \eqref{cbf road} encodes that the Ego vehicle remain on the road despite bounded steering control. We also have that \eqref{cbf speed} enforces the road speed limit, $L$, in m/s, and \eqref{cbf car} ensures that safety margins $s_x = \tau v_e\cos(\theta_e) + l_c$ and $s_y = w_c + 0.75$ between vehicles are observed, where $l_c$ and $w_c$ are the length and width of the vehicles in m. Then, $S = \cap S_i$, $\forall i \in \{1, 2, 3\}$.

In addition, Oncoming vehicles are known to obey the following pattern: the first vehicle has a time-headway of 24s with the Lead Vehicle, and subsequent Oncoming vehicles arrive in 30s intervals. Consequently, the Ego vehicle must complete the overtake within 24s to proceed at the outset, and within 30s to proceed after the first Oncoming vehicle. 

We now formally define the overtake problem.

\begin{Problem}
    Given the initial states, $z_e(0)$, $z_l(0)$, the time headway of an oncoming vehicle, $T_h$, and the set $\Theta$ to which the unknown parameter vector, $\theta$, belongs, determine whether it is safe for the Ego vehicle to overtake the Lead vehicle, i.e. whether there exist $z(t)$, $u_e(t) \in \mathcal{U} = \{(\omega_e, a_e) \; | \; \ubar{\omega} \leq \omega_e \leq \bar{\omega}, \; \ubar{a} \leq a_e \leq \bar{a}\}$ such that $z_e(t) \in S$, $\forall t \in [0,T]$, where $T$ is the upper bound on time to complete the overtake. If safe and $T \leq T_h$, design a control input, $u_e(t) \in \mathcal{U}$ for the given $z(0)$ such that the Ego vehicle overtakes the Lead vehicle.
\end{Problem}


\subsubsection{Control Formulation}
Just as in \cite{black2020quadratic}, we partition the problem into the following sub-problems:

\begin{enumerate}\label{list: subproblems}
    \item Ego Vehicle approaches Lead Vehicle
    \item Ego Vehicle merges into overtake lane
    \item Ego Vehicle advances beyond Lead Vehicle
    \item Ego Vehicle merges back into original lane
\end{enumerate}

\noindent We use the CLF-CBF-QP control framework presented in \eqref{CLF-CBF QP} to compute the control inputs, $u_1=\omega_e$ and $u_2=a_e$, pointwise-in-time where $q=3$ in accordance with $h_1(z)$, $h_2(z)$, and $h_3(z)$ in \eqref{cbf road}-\eqref{cbf car}. 
Our CLF is:
\begin{equation}
    V(z) = K_V(k_x\bar{x}^2 + k_y\bar{y}^2 + k_\theta\bar{\theta}^2 + k_v\bar{v}^2 - 1)
\end{equation}
where $\bar{x}=x-x_d$, $\bar{y}=y-y_d$, $\bar{\theta}=\theta-\theta_d$, and $\bar{v}=v-v_d$, and $z_d = [x_d \; y_d \; \theta_d \; v_d]^T$ is the desired state. We define the fixed-time convergence times for the four sub-problems as $T_1 = 3$, $T_2 = 5$, $T_3 = 7$, and $T_4 = 5$ respectively. 

{\small
\begin{table}[h!]
    \setlength{\tabcolsep}{3.5pt}
    \centering
    \caption{Overtake Parameters}\label{table overtake params}
    \begin{tabular}{|c|c|c|c|c|c|c|c|c|c|}
        \hline
        $\dot{x}$ & Val & QP & Val & CBF & Val & CLF & Val & $\dot{\hat{\theta}}$ & Val \\ \hline
        $M$ & 1994 & $Q_{0,0}$ & 1/$\bar{\omega}^2$ & $e_r$ & 0 & $K_V$ & 10$^{-5}$ & $k_e$ & 0.001 \\ \hline
        $f_{l,1}$ & 0.01 & $Q_{1,1}$ & 1/$\bar{a}^2$ & $e_l$ & 6 & $\mu$ & 5 & $T_e$ & 0.2 \\ \hline
        $f_{l,2}$ & 0.02 & $p_0$ & 5$\times10^{8}$ & $L$ & 30 & $\gamma_1$ & 0.8 & $\mu_e$ & 5 \\ \hline
        $\theta_1$ & 1 & $p_1$ & 1 & $l_c$ & 4.81 & $\gamma_2$ & 1.2 & $c_{1e}$ & 50 \\ \hline
        $\theta_2$ & 0 & $p_2$ & 1 & $w_c$ & 1.92 & $k_x$ & 0.0625 & $c_{2e}$ & 50  \\ \hline
         & & $p_3$ & 1 & $\tau$ & 1.8 & $k_y$ & 100 & $c_{2e}$ & 50  \\ \hline
         & & $\bar{\omega}$ & 0.175 & & & $k_\theta$ & 400 & $\gamma_{1e}$ & 0.8  \\ \hline
         & & $\bar{a}$ & 4890 & & & $k_v$ & 1 & $\gamma_{2e}$ & 1.2 \\ \hline
         & & & & & & & & $\ell_e$ & 100\\
        \hline 
    \end{tabular}
    Note: $Q_{i,j}$ denotes the value of the row $i$ column $j$ entry for the $Q$ matrix. Non-specified entries are uniformly zero.
    \vspace{-2mm}
\end{table}
}

\subsubsection{Results}

\begin{figure*}[!h]
    \centering
        \includegraphics[width=1\textwidth,clip]{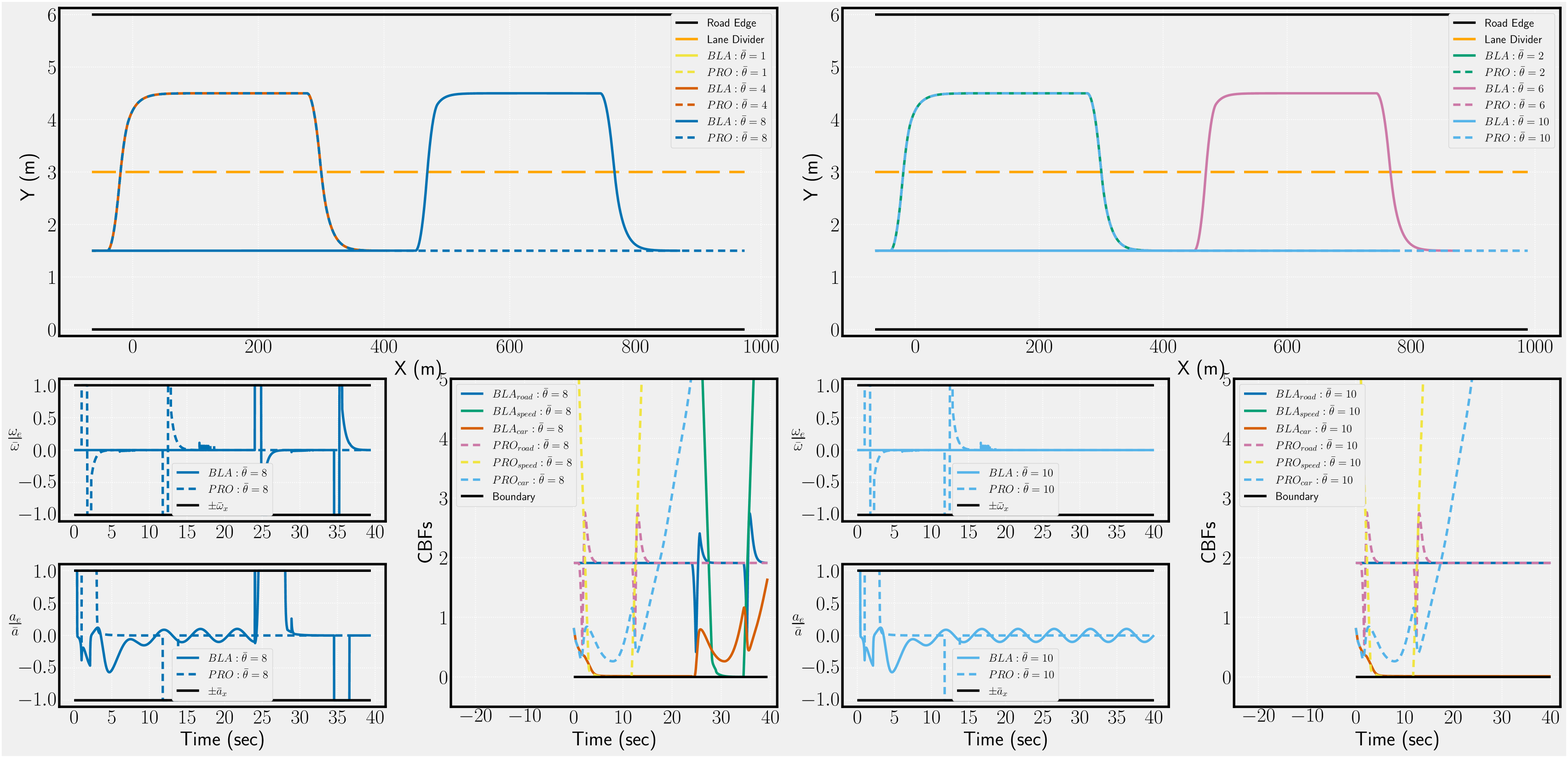}
    \caption{\small{Results for 6 different simulations of the overtake problem. The top row displays state trajectories. The bottom row contains control inputs and CBF trajectories for cases where \cite{black2020quadratic} must postpone the overtake maneuver (left) and cannot complete safely complete it (right).}} \label{fig: overtake}
    \vspace{-6mm}
\end{figure*}

The scenario was initialized as $x_e(0) = -64.8$, $y_e(0) = 1.5$, $\theta_e(0) = 0$, $v_e(0) = 24$, $x_l(0) = 0$, $y_l(0) = 1.5$, $y_\theta(0) = 0$, and $v_l(0) = 19$. For all considered sets of admissible parameters, $\Theta$, we set $\bar{\theta} = \bar{\theta}_1 = -\ubar{\theta}_1 = \bar{\theta}_2 = -\ubar{\theta}_2$, and chose $\bar{\theta} = 1, 2, 4, 6, 8, 10$. Table \ref{table overtake results} shows how the fixed-time horizon grows for BLA as $\bar{\theta}$ increases.

\begin{table}[h!]
    \centering
    \caption{Overtake Fixed-Time Horizons}\label{table overtake results}
    \begin{tabular}{|c|c|c|c|c|c|c|c|c|c|}
        \hline
        $\bar{\theta}$ & $T_{PRO}$ & $T_{BLA}$ & $t_{oncoming,1}$ & $t_{oncoming,i}$ \\ \hline
        1 & 20 & 20.74 & 24 & 24 + 30(i-1) \\ \hline
        2 & 20 & 21.31 & 24 & 24 + 30(i-1) \\ \hline
        4 & 20 & 22.67 & 24 & 24 + 30(i-1) \\ \hline
        6 & 20 & 24.42 & 24 & 24 + 30(i-1) \\ \hline
        8 & 20 & 26.80 & 24 & 24 + 30(i-1) \\ \hline
        10 & 20 & 30.38 & 24 & 24 + 30(i-1)\\
        \hline 
    \end{tabular}
    \vspace{-2mm}
\end{table}

As such, the PRO technique completes the overtake without delay for all parameter bounds, whereas the BLA controller appropriately proceeds immediately with the overtake when $\bar{\theta} = 1,2,4$, proceeds after the first oncoming vehicle has passed when $\bar{\theta} = 6,8$, and cannot guarantee a safe overtake when $\bar{\theta} = 10$. This is precisely the advantage of our proposed controller. Because it is guaranteed to adaptively learn the true parameters within fixed-time, it is able to successfully complete the overtake maneuver for all considered sets, $\Theta$. 


\section{Conclusion}
 In this study on the efficacy of various techniques for safe control under parametric model uncertainty, we presented a novel adaptation law that learns the uncertain parameters associated with a class of nonlinear, control-affine dynamical systems in fixed-time. We synthesized our parameter adaptation law with a robust, adaptive CBF-based controller in the form of a quadratic program, and provided an upper bound on the parameter estimation error as an explicit function of time. We then studied the performance of our method on a simple, 2D single integrator system in relation to several recent works from the literature and demonstrated that our contribution succeeds in navigating near unsafe regions where the others fail. We further illustrated the promise of our method in applications where a decision on whether to initiate a possibly unsafe maneuver is required, using the automobile overtake problem as a case study.

In the future, we intend to study cases for which the uncertain parameters are time-varying and an upper bound is not known a priori, as we recognize that these may have broader applicability to real-world scenarios.

\bibliographystyle{IEEEtran}
\bibliography{myreferences}



\end{document}